\documentclass{amsart}
\usepackage{amssymb}
\usepackage{latexsym}
\usepackage{pdfsync}

\newtheorem{lemma}{Lemma}

\theoremstyle{definition}
\newtheorem{definition}[lemma]{Definition}

\theoremstyle{remark}
\newtheorem{remark}[lemma]{Remark}

\numberwithin{equation}{section}

\newfont{\cyr}{wncyr8}
\newfont{\cyi}{wncyi8}

\newcommand{\cD}{{\mathcal D}}
\newcommand{\cE}{{\mathcal E}}
\newcommand{\cF}{{\mathcal F}}
\newcommand{\cG}{{\mathcal G}}

\newcommand{\cI}{{\mathcal I}}
\newcommand{\cL}{{\mathcal L}}
\newcommand{\cR}{{\mathcal R}}

\newcommand{\gD}{\Delta}

\newcommand{\gf}{\varphi}

\newcommand{\n}{^{-1}}

\begin{document}

\title{Transitive representations of inverse semigroups}
\author{Boris M. Schein}
\address{Department of Mathematical Sciences\newline\indent University of Arkansas\newline\indent
Fayetteville, AR 72701, USA}
\email{bschein@uark.edu}

%    \subjclass is required.

%    "Communicated by" -- provide editor's name; required.
\commby{Victoria Gould}

\begin{abstract}
While every group is isomorphic to a transitive group of permutations, the analogous property fails for inverse semigroups: not all inverse semigroups are isomorphic to transitive inverse semigroups of one-to-one partial transformations of a set. We describe those inverse semigroups that are.
\end{abstract}

\maketitle
 
 \section{Introduction and Preliminaries}

This paper consists of three sections. Section 1 explains Wagner's Problem raised in 1952 and solved here. Section 2 explains non-algebraic heuristic ideas that led to the solution. Section 3 contains the main result and all necessary proofs. 

We begin with definitions that make this paper reasonably self-contained. Different categories of readers may be more or less familiar with some of them. 

A {\em semigroup} is a nonempty set with associative multiplication. If $S$ is a semigroup and $sts = s,\; tst = t$ for $s,t \in S$, then $t$ is called an {\em inverse} for $s$. A semigroup is {\em regular} if each of its elements  has an inverse. A semigroup is {\em inverse} if each of its elements has a  {\em uniquely determined} inverse. If $S$ is inverse and $s \in S$, then $s\n$ denotes the unique inverse of $s$. Alternatively, inverse semigroups are precisely regular semigroups with commuting idempotent elements. A nonempty subset $T$ of an inverse semigroup $S$ is called an {\em inverse subsemigroup} of $S$ if $T$ is closed under multiplication and inversion, that is, $(\forall s,t)[ s,t\in T\Rightarrow st\in T]$ and $(\forall s)[ s\in T\Rightarrow s\n\in T]$. 

A {\em partial transformation} of a set $A$ is a mapping $\gf$ of a subset $B$  of $A$ into $A$. Thus $\gf(b)$ exists for all $b\in B$ and is not defined for $b\not\in B$. We will write $b\gf$ rather than $\gf(b)$. We say that $B$ is the {\em first projection} (or the {\em domain}) of $\gf$ and write $B={pr}_1\gf$. We say that $\gf$ is {\em one-to-one} if $a\gf= b\gf \Rightarrow a = b$ for all $a,b \in B$. Let $\cI_A$ denote the set of all one-to-one partial transformations of $A$. It is clear (and well known) that $\cI_A$ is closed under usual composition of partial transformations: $\gf,\psi\in \cI_A \Rightarrow \gf\circ\psi\in\cI_A$.  Usually we omit 
$\circ$ and write merely $\gf\psi$. Then $a(\gf\psi) =(a\gf) \psi$ for every $a\in A$ such that both $a\gf$ and 
$(a\gf)\psi$ are defined, that is, $a\in{pr}_1\gf$ and $a\gf= \gf(a)\in{pr}_1\psi$. In particular, $\gf\psi$ may be the empty partial transformation $\varnothing$. This empty partial transformation is obviously one-to-one and hence belongs to $\cI_A$.\footnote{We denote the empty partial transformation by the same symbol as the empty set because, for us, every partial transformation 
$\gf$ {\em is} a set. Namely, $\gf=\{(a,b)\in A\times A\mid a\gf=b\}$. From this point of view the empty partial transformation {\em is} the empty set.} Thus $\cI_A$ is a semigroup of partial transformations. Also, $\gf\in\cI_A \Rightarrow \gf\n\in\cI_A$. Here $\gf$ is the {\em inverse} transformation for $\gf$, that is, 
$(\forall a,b\in A)[a\gf\n=b\Leftrightarrow b\gf =a]$. Clearly, $\gf\n$ is the inverse for $\gf$ in the sense of the theory of inverse semigroups because $\gf\gf\n\gf= \gf$ and $\gf\n\gf\gf\n= \gf\n$ for every $\gf\in\cI_A$. We call $\cI_A$ the {\em symmetric inverse semigroup on $A$}. Inverse subsemigroups of $\cI_A$ are called {\em inverse semigroups of one-to-one partial transformations} of $A$. 

\begin{definition}\label{noble}
An inverse semigroup $\Phi$ of one-to-one partial transformations of a set $A$ is called {\em transitive} if, for every $a,b\in A$, there exists $\gf\in\Phi$ such that $a\gf =b$. An inverse semigroup is called {\em noble}\footnote{What is so noble in noble inverse semigroups? We need a term for them, and the word ``noble" has already been used in \cite{Sch6}.} if it is isomorphic to a transitive inverse semigroup of one-to-one partial transformations of a set. Observe that a trivial group and a two-element semilattice are both noble because they are isomorphic to transitive inverse subsemigroups of $\cI(A)$, where $A$ is a 
singleton $\{a\}$. We exclude a single-element group and a two-element semilattice from our further considerations.
\end{definition}

In other words, $\Phi$ is transitive when $\bigcup\Phi=A\times A$, where $\bigcup\Phi$ denotes the set-theoretical union of all elements of $\Phi$ (recall that each element of $\Phi$ is a special subset of $A\times A$). 

Observe that {\em permutations} of $A$ (that is, one-to-one mappings of $A$ onto itself) form an inverse subsemigroup $\cG_A$ of $\cI_A$. Clearly, $\cG_A$ is the symmetric group of permutations of $A$. 

\begin{remark}
By Cayley's Representation Theorem every group is isomorphic to a group of permutations, that is, to a subgroup of $\cG_A$ for a suitable set $A$. Moreover, every group is isomorphic to a transitive group of permutations. Thus every group is a noble inverse semigroup.
  
Groups are precisely the inverse semigroups with a single idempotent. We may conjecture  that a noble inverse semigroup must have few idempotents (to be ``closer" to a group.) Consideration of semilattices (a semilattice is an inverse semigroup where each element is idempotent) supports this conjecture. A  semilattice is noble only if it has no more than two elements. This is so because an idempotent one-to-one partial transformation of any set $A$ has the form 
$\gD_B= \{(b,b)\in A\times A\mid b\in B\}$, the identity mapping of a subset $B\subseteq A$ onto itself. Each ``point" $b\in B$ in the domain of $\gD_B$ is a fixed point, and thus our semilattice cannot move a point into a different point, that is, if it is transitive, then the cardinality of $A$ is at most 1. If $A=\varnothing$, then 
$\cI_\varnothing$ is a single-element semilattice, while if $A=\{0\}$, a singleton, then 
$\cI_{\{0\}} = \{\varnothing, \gD_{\{0\}}\}$, a two-element semilattice.  Yet the conjecture fails because the symmetric inverse semigroup $\cI(A)$ is transitive for any set $A$ and has plenty of idempotents. 

Although not every inverse semigroup is noble, nobility is quite common. By a theorem of Wagner\footnote{A more common transliteration of this name in English is ``Vagner". His father was German and they always spelled their name as usual: Wagner.  For his papers in Russian, some people in the West transliterated his name as ``Vagner" when the papers belonged to algebra, apparently unaware of his numerous earlier papers  (published by ``Wagner") in differential geometry, calculus of variations, nonholonomic mechanics and other fields. Thus Wagner was split into two persons.  It is like transliterating the name ``Poincar\'e" written in Russian as ``Puankare" in Roman alphabet.} 
(1952) and Preston (1954), every inverse semigroup $S$ is isomorphic to an inverse subsemigroup of 
$\cI_A$ for a suitable set $A$. But $\cI(A)$ is obviously noble, so every inverse semigroup can be embedded in a noble inverse semigroup. Moreover, it follows from the theory of representations of inverse semigroups by one-to-one partial transformations as constructed in \cite{Sch1} that every inverse semigroup is isomorphic to a subdirect product of noble inverse semigroups. 

In 1952 Wagner was the first to raise this problem: which inverse semigroups are isomorphic to transitive inverse semigroups of partial one-to-one transformations. We call it here {\em Wagner's Problem}. Our goal is to present a solution to this problem---to the 60-ieth birthday of inverse semigroups. Some of the ideas of this paper were announced by the author earlier (see \cite{Sch7}).

Wagner's problem attracted attention of many researchers. We give just a couple of references to previous research. Even before inverse semigroups appeared, a very special example of a (finite inverse) semigroup was shown to be ``noble" by R.R.~Stoll in his doctoral dissertation (see \cite{Stoll}). Also I.S.~Ponizovski\u\i{ }looked at Wagner's problem in many publications (see especially \cite{Poni}, its results were considerably strengthened in \cite{Sch4}). See also \cite{Sch3} etc.  
\end{remark}

\section{Infinitesimal Filters and Order of Their Magnitude}

Here we describe some heuristic ideas that permit solving Wagner's Problem. Readers interested in ``purely algebraic" solution stated and proved in the next section may skip Section 2 except the definitions 
in it.  Also, Section 2 contains a few technical lemmas. The readers not interested in the proof of the main result in the next section may skip these lemmas. 

The following motivation may be useful because the sources of inverse semigroup theory do not belong to algebra, they appeared in the torrent of new approaches of Niels Henrik Abel, Felix Klein, Sophus Lie, \'Elie Cartan, many other researchers and led to deep changes in our perceptions of function theory, geometry---in particular differential geometry, differentiable manifolds and differential topology---among other fields.  While ``everyone" knows that, sometimes people working in ``abstract" algebra had vague ideas of original motivations that generated ``their" classes of abstract algebraic systems, while people in ``non-algebraic" fields might know something about pseudogroups of diffeomorphisms not realizing that these ideas led to much simpler and more elegant algebraic systems (inverse semigroups). 

First we need a few standard definitions from inverse semigroup theory. 

On every semigroup one can introduce the so-called Greene's equivalence relations. For example, if  $s,t\in S$ for a semigroup $S$, we say that $s$ divides $t$ on the left if $s = t$ or $t = sr$ for some $r\in S$. Symbolically, we may write $s|_l\,t$. Clearly, $|_l$ is a quasi order relation on $S$ (that is, it is reflexive and transitive). It is not necessarily symmetric but we may define its symmetric part $\cL$ as follows: $(s,t)\in \cL$ precisely if $s|_l\,t$ and $t|_l\,s$.  Analogously, we can define the right analogues $|_r$ and $\cR$ of these relations. Both $\cL$ and $\cR$ are equivalence relations on $S$. 

Also $|_l$ and $|_r$ commute when multiplied as binary relations on $S$. In other words, 
$(|_r)\circ(|_l)= (|_l)\circ(|_r)$, that is,  
$(\exists u\in S^1)[s |_l u \wedge u |_r u]\Leftrightarrow(\exists v\in S^1)[s |_r v \wedge v |_l t]$.  Here $\wedge$ is the conjunction symbol, that is, $\wedge$ means ``and". Analogously, $\cL$ and $\cR$ commute, $\cL\circ\cR= \cR\circ\cL$, and hence $\cD=\cL\circ\cR$ is an equivalence relation on $S$. If  $S$ is an inverse semigroup, these definitions can be greatly simplified. For example,
$(s,t)\in\cD \Leftrightarrow(\exists u\in S)[ss\n = uu\n \wedge t\n t= u\n u]$. 

\begin{definition}\label{order}
For elements $s$ and $t$ of an inverse semigroup $S$ define $s\le t$ if and only if $s=ss\n t$.  It is known that $\le$ is a (partial) order relation on $S$. It is called the {\em natural} (or the {\em canonical}) order relation and can be defined in many equivalent ways. For example, $s\le t\Leftrightarrow s = ss\n t \Leftrightarrow s = st\n s \Leftrightarrow s = ts\n s \Leftrightarrow ss\n = st\n \Leftrightarrow ss\n=ts\n \Leftrightarrow s\n s = s\n t \Leftrightarrow s\n s= t\n s$. 

Other equivalent definitions can be given if one introduces a derived ternary operation in $S$. For any $s,t,u\in S$ define $[stu]= st\n u$ and call $[stu]$ the {\em triple product} of $s, t$ and $u$. Then $s\le t\Leftrightarrow s=[sst] \Leftrightarrow s=[sts] \Leftrightarrow s=[tss]$. 

The natural order relation is {\em stable} with respect to both multiplication and inversion in $S$, that is, $s_1\le t_1\wedge s_2\le t_2\Rightarrow s_1s_2\le t_1t_2$ and $s\le t\Rightarrow s\n\le t\n$ for all $s, s_1,s_2,t,t_1,t_2\in S$. 

A subset $H\subseteq S$ is called {\em closed} if $(\forall s,t)[ s\in H\wedge s\le t\Rightarrow t\in H]$. For every subset $T\subseteq S$ its {\em closure}  is the set $\overrightarrow{T} = \{t: (\exists s\in T)[s\le t]\}$. It is the smallest closed subset of $S$ that contains $T$. For $s\in S$ we write $\vec s$ rather than $\overrightarrow{\{s\}}$. For every $s\in S$ the closed set $\vec s$ is closed under the triple multiplication, that is, $t,u,v\in \vec s \Rightarrow [tuv]\in\vec s$. 
\end{definition}

If $\Phi$ is an inverse semigroup of one-to-one partial transformations of a set $A$ then the natural order relation on $\Phi$  is very transparent: $\gf\le\psi$ merely means $\gf\subseteq\psi$ for any $\gf,\psi\in\Phi$. Recall that we consider $\gf$ and $\psi$ as subsets of $A\times A$. Thus $\gf\subseteq\psi$ means that $\psi$ is an extension of $\gf$ to a greater domain (equivalently, $\gf$ is a restriction of $\psi$ to a smaller domain). In other words, $\gf\le\psi$ means ${pr}_1\gf\subseteq{pr}_1\psi$ and 
$a\psi =a\gf$ for all $a\in{pr}_1\gf$. We define the {\em second projection} (or the {\em range}) of $\gf$ as the subset ${pr}_2\gf \subseteq A$ of all $b\in A$ such that $(a,b)\in \gf$ for some $a\in A$. Divisibility relations are interpreted simply. For example, for 
$\gf,\psi\in\Phi$,
\begin{align*}
\gf|_l\,\psi  &\Leftrightarrow pr_1\gf \supseteq pr_1\psi,  &\gf\cL\psi \Leftrightarrow pr_1\gf &= pr_1\psi, \\ 
\gf|_r\,\psi  &\Leftrightarrow pr_2\gf \supseteq pr_2\psi,  &\gf\cR\psi \Leftrightarrow pr_2\gf &= pr_2\psi.
\end{align*} 

Each element $\gf$ of a symmetric inverse semigroup $\cI(A)$ can be viewed as the set-theoretical union of all one-to-one partial transformations $\{(a,b)\}$ of a set $A$ such that $(a,b)\in\gf$. Thus $\gf$ is the least upper bound (under the set-theoretical inclusion order) of all these $\{(a,b)\}$. 

Unfortunately, partial transformations $\{(a,b)\}$ do not have to be the elements of an arbitrary transitive inverse subsemigroup $\Phi$ of $\cI(A)$. Also, $\{(a,b)\}$ may possess abstract algebraic properties not shared by any elements of $\Phi$. For example, they are ``atoms" in the sense that every smaller partial transformation is empty, that is, the zero of  $\cI$. The product of transformations $\{(a_1,b_1)\}$  and 
$\{(a_2,b_2)\}$ is $\varnothing$ except when $b_1= a_2$. Even if $b_1= a_2$, the product of our transformations in the opposite order vanishes unless $b_2= a_1$, and hence both products are not 
$\varnothing$ only if our transformations are inverses each of the other.  If an inverse semigroup possesses something like these atoms, their existence may be exploited to express the fact that our inverse semigroup is noble (this line of thought was explored in \cite{Sch3}, \cite{Sch5} and \cite{Sch6}). Yet an inverse semigroup $\Phi$ may be transitive without possessing anything even remotely similar to ``atoms". 

We try another approach. In some sense $a\gf = b$ tells us something about the ``local" behavior of the function $\gf$. If $A$ is a topological space (or a differentiable manifold) we may be interested in the local behavior of $\gf$ in a neighborhood of the point $a$. We may say that functions $\gf$ and $\psi$ defined in certain neighborhoods of $a$ are equivalent with respect to $a$ if there exists a neighborhood $U$ of $a$ such that $U\subseteq pr_1\gf,\, U\subseteq pr_1\psi$ and the restrictions $\gf_{|U}$ and $\psi_{|U}$ of $\gf$ and $\psi$ to $U$ coincide. The equivalence class of $\gf$ with respect to $a$ is the set of all functions defined in neighborhoods of $a$ that are equivalent to $\gf$. Intuitively, this concept expresses the behavior of $\gf$ on an ``infinitesimally small" neighborhood of $a$. This idea leads us to the germs of the functions at a point and to the so-called ``pointwise localization."   

   Let $\Phi$ be a semigroup or an inverse semigroup of partial transformations of a topological space $A$ that may well be a differentiable manifold. Assume that all functions from $\Phi$ are defined on open subspaces or submanifolds of $A$, all of them are diffeomorphisms of a certain differentiability class, and all of them are ``locally one-to-one" (which means that for every $a\in A$ and any $\gf\in\Phi$ there is a neighborhood $U$ of $a$ such that $\gf_{|U}$ is one-to-one. If $\Phi$ is an inverse semigroup of partial one-to-one transformations, then $\Phi$ satisfies this requirement. For simplicity assume that we consider such inverse semigroups only.) 
   
   Suppose that $\gf$ and $\psi$ in $\Phi$ coincide on a sufficiently small neighborhood $U$ of point 
   $a\in A$. The equality  $\gf_{|U}= \psi_{|U}$ implies that 
   $\gf_{|U}\circ (\psi_{|U})\n= \gD_U$. Observe that 
$\psi_{|U}\subseteq\psi$, and hence $\gD_U= \gf_{|U}\circ (\psi_{|U})\n \subseteq \gf_{|U}\circ \psi$ and 
$U= pr_1\gD_U= \gf_{|U}\circ \psi$. We obtain $\gf_{|U}\circ \psi= \gD_U$. If $\chi$ in $\Phi$ is equivalent to 
$\gf$ and $\psi$, then all three of them coincide on a sufficiently small neighborhood of $a$. We can assume that $U$ is this neighborhood because if $U$ is ``too large", we can always make it smaller. Then  $[\gf\psi\chi]_{|U}= (\gf\circ \psi\n\circ \chi)_{|U}= \gf_{|U}\circ \psi\n \circ \chi= \gD_U\circ \chi= \chi_{|U}$. Thus  $[\gf\psi\chi]$ and $\chi$ are equivalent about $a$ (that is, with respect to $a$). It follows that the equivalence class of all $\chi\in\Phi$ that are equivalent to $\gf$ about $a$ is closed under the triple product! Also, this class is closed with respect to the natural order: if $a\in pr_1 \gf\in \Phi$ and $\gf\subseteq \psi$ for $\psi\in\Phi$ then $\gf$ and 
$\psi$ coincide on a sufficiently small neighborhood of $a$. It is not difficult to see that in this situation the ``germs" of the functions $\gf\in\Phi$ at $a\in A$ are subsets of $\Phi$ that are closed under both the triple multiplication and the natural order relation. 

In the case when $\Phi$ is a semilattice the concept of germs simplifies. In such a case a germ is a subset of $\Phi$ closed both under the natural order and the ordinary (rather than the triple) multiplication (see Lemma \ref{filt3}(3) below). In this case a germ usually appears under a different name, that of a {\em filter}. Interestingly, the ``ideology" of filters originally appeared in a not totally dissimilar situation when Henri Cartan and Andr\'e Weil wanted to free topology from the constraints of countability and write a ``modern" course of analysis. Cartan invented filters but called them merely ``boum!" (French  for ``bang!"). In a  formal publication \cite{Car} they used a more ``civilized" name. 

\smallskip
What does it mean that $\Phi$ is transitive? From the ``germ" point of view this means that for all $a,b\in A$ there exists a sufficiently small neighborhood $U$ of $a$ and a function $\gf\in\Phi$ such that 
$a\gf_{|U}= b$. Is not this definition of transitivity more complicated than our previous definition where we used merely $a\gf=b$? Yes because our new equality with $\gf_{|U}$ looks more complicated. No because the existence of $\gf$ moving $a$ to $b$ was not an ``abstract" algebraic property of $\Phi$ while the same equality in the ``germ" form can be defined in the language of ``abstract" algebra. A property is ``abstract" when it is invariant under isomorphisms. For example, two inverse semigroups $\Phi$ and $\Psi$ may be isomorphic as abstract semigroups while one of them may be transitive and the other one not transitive. However, germs may be defined for abstract inverse semigroups, and there is a hope that in this language noble inverse semigroups may possess special abstract properties not shared by all inverse semigroups. We shall see that this is indeed the case.

Thus if $\Phi$ is transitive it has a ``transitive set" of germs, that is, $\Phi$ has a subset $\cF$ of germs such that for all $a,b\in A$ there is a germ around $a$ that moves $a$ to $b$. Every function $\gf\in\Phi$ defined in a neighborhood of $a$ belongs to a germ that moves $a$ to some fixed point of $A$.  Thus $\Phi$ must possess a transitive set of ``very small" functions, the functions with a small domain. 

Now we need to use another idea that appeared in other parts of mathematics. What is ``small"? If $x < y$ then $x$ is smaller than $y$. But what is ``small"? That depends. It is much easier to define when two objects have ``the same size" or ``the same magnitude" rather than when one of them is smaller than the other. 

Two sets, $X$ and $Y$, are of the same magnitude when there is a bijection $\gf$ of $X$ onto $Y$. So, from the point of view of inverse semigroups, let $\Psi$ be the inverse semigroups of all possible bijections between all possible sets. We can define the same cardinality, the same ``magnitude" using the elements of $\Psi$. Of course, this $\Psi$ never exists (if it does, then we immediately hit hard set-theory paradoxes). But we are on a ``naive" level. Can't we dream? We can---but there is another difficulty. Earlier we observed that not every transitive inverse semigroup $\Phi$ has the ``atoms" $\{(a,b)\}$ that $\cI(A)$ has. And we are equally right saying that not every transitive inverse semigroup is our inverse ``mega-semigroup" $\Psi$. We are given $\Phi$, it is our ``world" and we have to decide what has the same magnitude with respect to $\Phi$. In other words, subsets $B$ and $C$ of the set $A$ have the same magnitude with respect to $\Phi$ if there is 
$\gf\in\Phi$ that is a bijection of $B$ onto $C$. If $\gf\in\Phi$ maps $B$ onto $C$ then $\gf\gf\n= \gD_B$, $\gf\n\gf= \gD_C$ and so $\gf\gf\n$ and $\gf\n\gf$ have the same magnitude ``from the point of view" of $\Phi$. 
Recall our discussion of the $\cD$-relation in the beginning of this section. We can say that $\gf$ and $\psi$ have the same magnitude in $\Phi$ when these elements of $\Phi$ are $\cD$-related. 

   And yet we are far from salvation. This is because, in the context of germs, we have to compare not 
$\gf$ and $\psi$ themselves but the germs to which they belong. The germ doesn't have to be a function. It is a function defined on an ``infinitesimally small" neighborhood of $a\in A$ and $\Phi$ may contain no such functions. Suppose $G_1$ is a germ of $\Phi$ that moves $a_1$ to $b_1$ and $G_2$ is a germ that moves $a_2$ to $b_2$. We can say that $G_1$ and $G_2$ have the same magnitude with respect to $\Phi$ if there exist $u,v\in\Phi$ such that $u$ is defined in a neighborhood of $a_2$ and $a_2u=a_1$. Likewise, 
$v$ is defined in a neighborhood of $b_2$ and it moves $b_2$ to $b_1$ and also 
$uG_1v\n\subseteq G_2$ and $u\n G_2 v\subseteq G_1$. 

Now we are ready to introduce more definitions. Following Ockham's Razor principle (``Don't multiply the entities without necessity") we use the term ``filter" rather than ``germ"). 

\begin{definition}\label{filt2} Let $S$ be an inverse semigroup.  A nonempty and proper subset $F\subset S$ is called a {\em filter} if $F$ is closed and also closed under the triple multiplication. For each nonzero 
$s\in S,\, \vec s$ is a filter called the {\em principal} filter generated by $s$.

Filters $F_1$ and $F_2$ of $S$ are {\em of the same magnitude with respect to} $S$ (denoted as $F_1\sim  F_2$) if there exist $u,v\in S$ such that $uF_1v\n\subseteq F_2$ and $u\n F_2v\subseteq F_1$.

Inverse subsemigroups $H$ and $K$ of $S$ are called {\em conjugate} if  $H$ and $K$ are filters of $S$ and $H\sim K$. 
\end{definition}

Observe that if $s$ is the zero of $S$ then $\vec s= S$ cannot be a filter because it is an improper subset of $S$.

Lemma \ref{filt1} describes certain properties of the natural order relation in inverse semigroups. If $S$ is an inverse semigroup and $H$ and $K$ are subsets of $S$ then $HK=\{s\in S|\, (\exists h\in H, k\in K)[s=hk]\}$ and $H\n= \{s\in S|\, (\exists h\in H)[s=h\n]\}$. 

\begin{lemma}\label{filt1}  \begin{enumerate}
\item For any subsets $H$ and $K$ of an inverse semigroup $S$,\\ 
$\overrightarrow{HK}= \overrightarrow{\overrightarrow{H}K}= \overrightarrow{H\overrightarrow{K}}$.
\noindent \item  For all $s,t\in S$, $s\le t$ if and only if $\vec s\supseteq \vec t$. In particular, 
$s\mapsto\vec s$ is a natural one-to-one correspondence between non-zero elements of $S$ and principal filters. ccccccccc
\end{enumerate}
\end{lemma}

\begin{proof} (1). Clearly, $HK\subseteq\overrightarrow{H}K$ and so 
$\overrightarrow{HK}\subseteq\overrightarrow{\overrightarrow{H}K}$. Also, if $hk\in\overrightarrow{H}K$ for some $h\in\overrightarrow{H}$ and $k\in K$ then $h_1\le h$ for some $h_1\in H$ and so $h_1k\le hk$. Thus $hk\in\overrightarrow{HK}$ and
$\overrightarrow{\overrightarrow{H}K}\subseteq\overrightarrow{HK}$. That proves 
$\overrightarrow{HK}= \overrightarrow{\overrightarrow{H}K}$. The equality 
$\overrightarrow{HK}= \overrightarrow{H\overrightarrow{K}}$ is proved similarly.

(2) If $s\le t$ and $u\in \vec t$ then $t\le u$ and so $s\le u$. It follows that $u\in\vec s$, and hence 
$\vec s\supseteq\vec t$. Conversely, if $\vec s\supseteq\vec t$ then $t\in\vec t$ and so $t\in\vec s$ which means that $s\le t$. 
 If $\vec s= \vec t$ for $s,t\in S$ then $s\le t$ and  $t\le s$, so that $s=t$. 
\end{proof}

\begin{lemma}\label{filt3}  The relation $\sim$ is an equivalence relation. The following properties $(1)$ and $(2)$ are equivalent for any filters $F_1$ and $F_2$ of an inverse semigroup $S$:

\begin{enumerate} 

\item  $uF_1v\n\subseteq F_2$ and $u\n F_2v\subseteq F_1$  for some $u,v\in S$, that is, $F_1$ and $F_2$ are of the same magnitude with respect to $S$; 

\item $\overrightarrow{uF_1v\n}= F_2$ and $\overrightarrow{u\n F_2v}= F_1$.

In particular,  the principal filters $\vec s$ and $\vec t$ are of the same magnitude if and only if $s$ and $t$ are $\cD$-equivalent elements of $S$. 

\item Filters $H$ that contain idempotents are precisely closed (that is, $\vec H= H$) inverse subsemigroups of $S$. In particular, $\vec s$ is a proper inverse subsemigroup of $S$ exactly when $s$ is a non-zero idempotent of $S$. Closed inverse subsemigroups $H$ and $K$ of $S$ are of the same magnitude in $S$ if and only if they are conjugate. This is so precisely when  $\overrightarrow{uHu\n}= K$ and 
$\overrightarrow{u\n Ku}= H$ for some $u\in S$ or, equivalently, $\overrightarrow{uHu\n}= K$ for some $u$ such that $u\n u\in H$. 
\end{enumerate}
\end{lemma} 

\begin{proof}  Indeed,  If $f,h\in F$ for a filter $F$ then $(ff\n)h(f\n f)= [[ffh]ff]\in F$ because $F$ is closed under the triple product. Thus $F\sim F$. 
If $F_1\sim F_2$ for any filters $F_1$ and $F_2$, then $F_2\sim F_1$ follows immediately from the definition of $\sim$. If $F_1\sim F_2$ and $F_2\sim F_3$ then, for some $u,v,w,x\in S$, 
\begin{align*} uF_1v\n&\subseteq F_2, &u\n F_2v&\subseteq F_1, 
&wF_2x\n&\subseteq F_3, &w\n F_3v&\subseteq F_2.
\end{align*}
Then $wuF_1v\n x\n\subseteq wF_2x\n\subseteq F_3$, similarly $u\n w\n F_3xv \subseteq F_1$, and so $F_1\sim F_3$.

Obviously (2) implies (1). If (1) holds then $F_1\subseteq u\n u F_1 v\n v \subseteq u\n F_2v\subseteq F_1$ and so $\overrightarrow{u\n F_2v}= F_1$. Similarly, $\overrightarrow{uF_1v\n}= F_2$. So (2) holds. 

Now let $\vec s\sim \vec t$. Then $\overrightarrow{u\vec s v\n}= \vec t$ and 
$\overrightarrow{u\n\vec t v} =\vec s$ for some $u,v\in S^1$.  It follows from Part (1) of Lemma \ref{filt1} that 
$\overrightarrow{usv\n}= \overrightarrow{u\vec s v\n}$. Therefore, $\overrightarrow{usv\n}= \vec t$ and so $usv\n=  t$. Similarly,  $u\n tv= s$. It follows from $s= u\n tv= u\n(usv\n)v = (u\n u)s(v\n v)$ that $s =(u\n u)s$, and so $(s,us)\in\cL$. Similarly, $t= t(v\n v)$ and so $(t,tv\n)\in\cR$. However $us= usvv\n= tv\n$ and thus 
$(s,t)\in\cL\circ\cR= \cD$. Conversely, if $(s,t)\in\cD$ then $s= t$ or $s\n s= uu\n$ and $u\n u= tt\n$ for some $u\in S$. If $s= t$ then $s\n st = stt\n$, and hence  $\vec s$ and $\vec t$ are of the same magnitude. If 
$s\n s= uu\n$ and $u\n u= tt\n$ for some $u$, then $s\n su = suu\n$, and hence $\vec s$ and $\vec u$ are of the same magnitude. Similarly,  $utt\n = u\n ut\n$, and hence $\vec u$ and $\vec t$ are of the same magnitude. Thus $\vec s$ and $\vec t$ are of the same magnitude.

(3) Indeed, let $H$ be a closed proper inverse subsemigroup of $S$. Is $0\in S$, then $0\notin H$. If 
$s,t,u\in H$ then $t\n\in H$, and hence $[stu]= st\n u\in H$. Thus $H$ is a filter. Conversely, if $H$ is a filter that contains an idempotent $e\in H$ then, for all $s,t\in H$ we have $set=se\n t= [set]\in H$ and also 
$set\le st$. So $st\in H$. Also, for every $s\in H,\, es\n e=[ese]\in H$ and $es\n e\le s\n$, so that $s\n\in H$. So $F$ is a closed inverse subsemigroup of $S$. 

Suppose that $H$ and $K$ are closed inverse subsemigroups of $S$. By Part (2) of this Lemma, $H\sim K$ means that  $\overrightarrow{uHv\n}= K$ and $\overrightarrow{u\n Kv}=H$ for some $u$ and $v$. Without loss of generality we may assume that $u=v$. Indeed, if $e$ is idempotent in $K$ then $u\n ev\in H$. Thus 
$(u\n ev)(u\n ev)\n = u\n evv\n eu \le u\n u\in H$.  Likewise, $uu\n\in K$. So $H$ and $K$ contain 
$\cD$-related idempotents (for example, the idempotents $u\n u$ and $uu\n$). 

Also, $u\n ev\le u\n v$ and so $u\n v\in H$. Similarly, $uv\n\in K$. Thus 
$$\overrightarrow{u\n Ku}= \overrightarrow{u\n Kvv\n u}= \overrightarrow{\overrightarrow{u\n Kv}v\n u}= \overrightarrow{Hv\n u}= H.$$ 
Analogously, $\overrightarrow{uHu\n}= K$. Likewise, $H\sim K$ if an only if $uHu\n\subseteq K$ and 
$u\n Ku\subseteq H$ for some $u\in S$. Finally, if $\overrightarrow{uHu\n}= K$ for some $u$ such that 
$u\n u\in H$, then $u\n Ku \subseteq\overrightarrow{u\n uHu\n u}\subseteq H$ and, by Part (1) of this Lemma, $H\sim K$. 
\end{proof}

The filters of an abstract inverse semigroup $S$ are subsets of $S$ and thus they are (partially) ordered by the set-theoretical inclusion $\subseteq$. Yet, rather than $\subseteq$, we will use the opposite order $\supseteq$ and now we explain motivation for that. 

As we have seen in Lemma \ref{filt1}, $s\le t\Leftrightarrow \vec s\supseteq\vec t$ for all $s$ and $t$ in $S$.  We can imagine $S$ as its order diagram: the elements of $S$ are vertices, the greater (with respect to $\le$) elements of $S$ are situated above the smaller elements. Then $s\in S$ is a vertex and $\vec s$ is the ``shadow" cast by $s$ when the source of light is below $S$ and the rays of light go upwards.  Thus 
$\vec s$ is a sort of a ``cone", a subset of $S$ consisting of the elements $\ge s$. The smaller (that is, the lower) the element $s$ is, the bigger shadow it casts on $S$. 

Now suppose that $S$ is an inverse subsemigroup of a bigger inverse semigroup $T$. If $F\subseteq T$ is a filter  in $T$, it {\em induces} a filter $F\cap S$ in $S$ (it is easy to check that, indeed, $F\cap S$ is a filter of $S$ provided that it is a nonempty proper subset of $S$). If $a\in T$ then the principal filter $\vec a$ of $T$  induces a filter $F_a=\vec a\cap S$ on $S$ (again, provided that $\vec a\cap S\ne\varnothing$ and 
$S\not\subseteq \vec a$).  Then $F_a$ is a part of the shadow of $a$, namely, $F_a$ is the shadow that $a$ casts on $S$ only. It turns out that  for each filter $F$ in $S$ there is a $T$ such that $F$ is induced by a {\em principal} filter of  $T$. Moreover, $T$ can be chosen so that {\em all} filters of $S$ are induced by  principal filters of the same $T$ (see \cite{Sch8} and \cite{Sch9}). Actually, we may organize the set of all filters of $S$ into an inverse semigroup with $S$ being isomorphically embedded into it by $s\mapsto \vec s$, the filters would be $\cD$-equivalent in that inverse semigroup when they are of the same magnitude etc. Again, see \cite{Sch8} and \cite{Sch9}. 

Every filter of $S$ is a shadow cast onto $S$ by an element of a bigger inverse semigroup $T$.  If $a,b\in T$ and $a\le b$ in $T$ then, obviously, $F_b\subseteq F_a$. Thus the smaller (the lower) is an element of $T$, the bigger shadow it casts onto $S$.  

Recall that $\le$ is the canonical order relation of $S$. If a subset $T\subseteq S$ has the least upper bound $u$ with respect to $\le$, we write $u= {\bigvee} T$.  If we replace the elements by their principal filters $\vec s$, then $\vec u= \bigcap_{s\in T}\vec s$, or, equivalently, 
$\overrightarrow{\bigvee T}= \bigcap_{s\in T}\vec s$, that is, the l.u.b. $\bigvee$ is replaced by the set-theoretical {\em intersection} $\bigcap$ rather than by the union $\bigcup$. This occurs because $\le$ on the elements of $S$ is replaced by the {\em inverse} set-theoretical inclusion $\supseteq$. 
\begin{definition}\label{dist} 
(1) A subset $B$ of $S$ is called a {\em basis} of $S$  if each element of $S$ is the least upper bound of an appropriate subset of $B$.

(2) A basis $B$ of $S$ is called {\em uniform} if all elements of $B$ are of the same magnitude (that is, any principal filters $\vec b_1$ and $\vec b_2$ for $b_1,b_2\in B$ are of the same magnitude).
\end{definition}

As proved in \cite{Sch6}, if  an inverse semigroup has a uniform basis, then it is noble, that is, isomorphic to a transitive subsemigroup of $\cI(A)$ for some set $A$. Yet an infinite transitive subsemigroup of $\cI(A)$ may have no uniform basis.
\smallskip

Now suppose that $\Phi$ is a transitive inverse subsemigroup of $\cI(A)$ and $a,b\in A$. Since $\Phi$ is transitive, the set $H_a^b=\{\gf\in\Phi: a\gf=b\}$ is not empty. Obviously, $H_a^b$ is closed because if 
$\gf\in H_a^b$ and $\gf\le\psi$, then $(a,b)\in\gf\subseteq\psi$, and hence $\psi\in H_a^b$. Also, if 
$\gf,\psi,\chi\in H^b_a$ then $a\gf\psi\n\chi = b\psi\n\chi=a\chi=b$ and $[\gf\psi\chi]= \gf\psi\n\chi \in H_a^b$, that is, $H_a^b$ is closed under the ternary multiplication. Thus $H_a^b$ is a filter in $\Phi$. If $a_1,a_2,b_1,b_2\in A$ then $a_1s=a_2$ and $b_1t=b_2$ for some $s,t\in\Phi$ so that 
$sH_{a_2}^{b_2}t\n\subseteq H_{a_1}^{b_1}$ and $s\n H_{a_1}^{b_1}t\subseteq H_{a_2}^{b_2}$, that is, $H_{a_1}^{b_1}\sim H_{a_2}^{b_2}$. Therefore all filters $H_a^b$ are of the same magnitude for all $a,b\in A$.
Each element $\gf\in\Phi$ is completely determined by the set of filters $\cF_{\gf}= \{H_a^b:\, \gf\in H_a^b\}$ that contain $\gf$ because $\vec\gf=\bigcap\cF_\gf$. This explains the following definition.

\begin{definition}\label{filt6} An infinitesimal basis of an inverse semigroup $S$ is an indexed  set 
$\{F_i: i\in I\}$ of filters of $S$ such that all filters $F_i$ are of the same magnitude and they form a basis of $S$ in the sense that, for each $s\in S$, $\vec s= \bigvee\{F_i: \vec s\subseteq F_i\}$ or, equivalently, 
$\vec s= \bigcap\{F_i: s\in F_i\}$.
\end{definition}

Uniform bases from Definition \ref{dist} form a special case of those from Definition \ref{filt6}. Uniform bases consist of ``small" elements of $S$ because each element of $S$ is the l.u.b. of some elements from the basis.  Now, instead of small elements or their principal filters, we can use even ``smaller" filters. That explains why these new bases are called infinitesimal. If a set of filters forms a basis of $S$, these ``infinitesimal filters" don't have to have the same magnitude. However, in our definition, we require all of them to be of the same magnitude, have the same degree of ``infinitesimality". 

Also observe that if $S$ has a zero $0$, then $0$ is the l.u.b.~of an empty set of filters.
\smallskip

It turns out that an infinitesimal basis of $S$ is uniquely determined by each of its filters. Indeed, if $F$ is a filter from an infinitesimal basis $B$ and $K$ is another filter of $B$, then there exist $u,v\in S$ such that $\overrightarrow{uFv\n}= K$ and $\overrightarrow{u\n Kv}=F$. Thus $K$ is completely determined by $F,\ u$ and $v$. It remains to see which of the elements $u$ and $v$ will produce another filter from $B$. 
 
 As we saw in the second and third paragraphs of the proof of Lemma \ref{filt3}.(3), $u\n uFv\n v\subseteq F$ for our $u$ and $v$. For every $s\in F,\, u\n usv\n v\le u\n us\le s$, and hence 
 $u\n uFv\n v\subseteq \overrightarrow{u\n uF}\subseteq\vec F= F$. Thus $u\n uF\subseteq F$. Analogously, $Fv\n v\subseteq F$. Conversely, if $u\n uF\subseteq F$ and $Fv\n v\subseteq F$, then $u\n uFv\n v\subseteq\overrightarrow{u\n uFv\n v} =F$ and the filter $\overrightarrow{uFv\n}$ is of the same magnitude as $F$. 
 
Observe that the filters $\overrightarrow{uF}$ and $\overrightarrow{Fv\n}$ are of the same magnitude as 
$F$.  If $s\in F$ then $Fs\n s\subseteq F$ becase $F$ is closed under the ternary multiplication and, for every $t\in F$, we have $ts\n s= [tss]\in F$. It follows that $\overrightarrow{Fs\n}\sim F$. 

The filter $\overrightarrow{Fs\n}$ contains idempotents. For example, $ss\n\in \overrightarrow{Fs\n}$.

If $S$ is a group then closed inverse subsemigroups of $S$ are precisely the subgroups of $S$ and 
$H\sim K$ means that $uHu\n= K$. That gives a new (or rather an old) meaning to the term ``conjugate" we use. Observe that here we don't need the symmetric part $u\n Ku=H$ while, in the case of arbitrary inverse semigroups $S$ this symmetric part is needed. It can be replaced by a simpler property: $H\sim K$ if and only if $uHu\n= K$ and $u\n u\in H$. 

\begin{definition}\label{filt4} For a closed inverse subsemigroup  $H$ of $S$,  let $\cE_H$ be the set of all closed inverse subsemigroups conjugate to $H$ in $S$. We call $\mathcal E_H$ an {\em idempotent infinitesimal basis in} $S$ when every idempotent $e\in S$ is the l.u.b. of an appropriate subset of $\cE_H$. In such a case $H$ is called an {\em infinitesimal inverse subsemigroup}. 
\end{definition}

Observe that we require $\cE$ to be a basis not in  $E(S)$ but {\em in} $S$, that is, for every $e\in E(S)$, $\vec e= \bigcap\{F\vert e\in F\in\cE\}$. Here $\vec e$ is considered in $S$, that is, it may contain non-idempotent elements of $S$.
 
\section{Main Result}
 
\noindent{\bf Main Theorem}.  \label{main}{\em For every inverse semigroup $S$ the following properties are equivalent:
 
\begin{enumerate}
 
 \item $S$ is noble, that is, isomorphic to a transitive inverse semigroup of one-to-one partial transformations of an appropriate set;
 
  \item $S$ possesses an infinitesimal basis;

   \item $S$ possesses an infinitesimal closed inverse subsemigroup.

  \end{enumerate}}
 
 \begin{proof} $(1)\Rightarrow(2)$.  Suppose that $f$ is an isomorphism of $S$ onto a transitive subsemigroup $f(S)$ of the symmetric inverse semigroup $\cI(A)$ on an appropriate set $A$. For all 
 $(a,b)\in A\times A$ define a subset $H^a_b$ of $S$ as follows: $s\in H^a_b \Leftrightarrow af(s)=b$, that is, the partial one-to-one transformation $f(s)$ moves $a$ into $b$. In the motivation for Definition \ref{filt6} (the paragraph immediately before this definition) we proved that $\{H^a_b|\; (a,b)\in A\times A\}$ is an infinitesimal basis for $S$. 
 
$(2)\Rightarrow(3)$.  By Lemma \ref{filt3}.(3) the filters that contain idempotents of $S$ are precisely closed inverser subsemigroups of $S$. 
  
  Suppose that $S$ has an infinitesimal basis $\cF=\{F_i\}_{i\in I}$.  If $e\in S$ is idempotent then $e$ is the l.u.b.~of those $F_i$ that contain $e$.  If there is no such $F_i$ than $e$ is the l.u.b.~of an empty subset of $\cF$, that is, $e$ is the least element of $S$. Thus $e=0$. It follows that if $e\ne 0$ then $e\in F_i$ for some $i\in I$. This $F_i$ contains an idempotent and so it is a closed inverse subsemigroup of $S$. 

All $F_i$'s with idempotents are conjugate because they belong to the same infinitesimal basis of $S$. 
It remains to prove that $H$ is infinitesimal. Recall that $\cE_H$ denotes the set of all closed inverse subsemigroups conjugate to $H$ in $S$. Thus each of these $F_i$'s is an infinitesimal closed inverse subsemigroup of $S$. 

$(3)\Rightarrow(1)$. Suppose that $H$ is an infinitesimal closed inverse subsemigroup of $S$ and $\cF$ is the set of all filters $F\in S$ such that $F\sim H$, that is,  $F= \overrightarrow{uHv\n}$ and $H=\overrightarrow{u\n Fv\n}$ for some $u,v\in S^1$.  It follows from our comments after Definition \ref{filt6} that $\cF$ consists of all $F$ such that $F= \overrightarrow{uHv\n}$ for $u,v\in S$ such that $u\n u, v\n v\in H$. 

For each $s\in S$ define the following binary relation $f(s)\subseteq \cF \times \cF$ between the elements of $\cF$.  For any $F_1, F_2\in \cF$ let $(F_1, F_2)\in f(s)$ precisely when $F_1s\subseteq F_2$ and $F_2s\n\subseteq F_1$. By Lemma \ref{filt3}.(3) this means that $\overrightarrow{F_1s}= F_2$ and $\overrightarrow{F_2s\n} =F_1$. Clearly, $(F_1, F_2)\in f(s)$ if and only if $(F_2, F_1)\in f(s\n)$, that is, $f(s\n)= f(s)\n$. It follows that $f(s)$ is a partial one to one transformation of $\cF$. If $(F_1, F_2)\in f(s)$ and 
$(F_2, F_3)\in f(t)$, then $$F_1f(st)= \overrightarrow{F_1st}= \overrightarrow{\overrightarrow{F_1 s}t}=
(F_1f(s))f(t)= F_2f(t)= F_{3}$$ and similarly  
$F_3(st)\n= (F_3f(t)\n)f(s)\n \subseteq F_2s\n= F_1$, that is,\\$f(st)\subseteq f(s)\circ f(t)$. Also, if 
$(F_1,F_3)\in f(s)\circ f(t)$ for some $F_1, F_3\in\cF$ then $(F_1, F_2)\in f(s)$ and $(F_2, F_3)\in f(t)$ for some $F_2\in \cF$. Thus $f(s)\circ f(t)\subseteq f(st)$. It follows that $f(st)=f(s)\circ f(t)$ for all $s,t\in S$ and the mapping $s\mapsto f(s)$ is a homomorphism of $S$ into $\cI(\cF)$.

To prove that this representation is faithful,  suppose that $f(s)= f(t)$ for some $s,t\in S$. Then 
$f(ss\n)= f(s)\circ f(s\n)= f(s)\circ f(s)\n= f(s)\circ f(t)\n= f(st\n)$. Here $ss\n$ is an idempotent of $S$, and hence $f(ss\n)$ is an idempotent of $\cI(\cF)$. So $(F,F)\in f(ss\n)$ for $F\in\cF$ precisely when $Fss\n\subseteq F$ or, equivalently, $ss\n\in F$. But $ss\n\in F$ means that $F$ is a closed inverse subsemigroup of $S$ because $F$ is a filter that contains an idempotent. Recall that $F\in\cF$ precisely when $F\sim H$, which means that $F$ is a conjugate of $H$. Thus $f(ss\n)$ is an identity mapping of the set of all closed inverse subsemigroups of $S$ that are conjugates of $H$ and contain the idempotent $ss\n$. But $H$ is an infinitesimal closed inverse subsemigroup of $S$, which by assumption (3) means that $ss\n$ is the l.u.b. of a certain subset of infinitesimal closed inverse subsemigroups conjugate to $H$. Since $f(ss\n)= f(st\n)$, then $st\n$ is an element of $S$ that, too, is a l.u.b. of the same set of closed inverse subsemigroups of $S$. Thus $ss\n= st\n$. It follows that $s=ss\n s= st\n s= [sts]$, and so $s\le t$. 

Interchanging $s$ with $t$ in the equality $f(s)= f(t)$ we obtain $t\le s$. Thus, if $f(s)=f(t)$ then $s=t$, that is, the isomorphism $f:S \to \cI(\cF)$ is faithful. 

By our definition of $\cF$, if $uu\n\in H$ for some $u\in S$ , then $F=\overrightarrow{Hu}\in\cF$. Thus 
$Hf(u)=F$, that is, $f(u)$ transforms $H$ into $F$. But $f(s)\n= f(s\n)$ and so $f(u\n)$ sends $H$ to 
$\overrightarrow{Hu\n}\in\cF$. Also, 
$(\overrightarrow{uH})\n =  \overrightarrow{H\n u\n}=  \overrightarrow{Hu\n}\in\cF$ because $H\n= H$. Thus  $f(u)=f(u\n)\n$ transforms $\overrightarrow{uH}$ into $H$. 

Now let $F$ be an arbitrary filter from $\cF$. By Lemma \ref{filt3}.(2)  $F=\overrightarrow{uHv\n}$ for certain $u,v\in S$ such that $u\n u, v\n v\in H$. As we have just seen, then  $f(u)$ moves $H$ to $\overrightarrow{uH}$ and so $f(uv\n)= f(u)\circ f(v\n)$ moves 
$\overrightarrow{uH}$ to $\overrightarrow{\overrightarrow{uH}v\n}= \overrightarrow{uHv\n}= F$.

Now suppose that $F_1,F_2\in\cF$. Then there exist $s_1,s_2\in S$ such that  $Hf(s_i)=F_i$ for $i=1,2$ 
and so $F_1f(s_1\n s_2)=F_2$. Therefore, $s\mapsto f(s)$ is a transitive representation of $S$ in $\cI(\cF)$.
\end{proof} 


\begin{thebibliography}{00}

\bibitem{Car} H.~Cartan, 
{\em Th\'eorie des filtres} and {\em Filtres et ultrafiltres}. C.r.~Acad.~sci. de Paris {\bf 105}(1937), 595--698 and 777--779.

\bibitem{Poni} 
I. S. Ponizovski\u\i, On representations of inverse semigroups by partial one-to-one transformations. Izvestiya Akademii Nauk SSSR, seriya matem. {\bf 28}(1964), no. 5, 989--1002 [Russian]; MR {\bf 30}, 179; R\v ZMat 1965, 3A237.

\bibitem{Sch1}
Boris M. Schein,  Representations of generalized groups. 
 Izvestiya Vys\v sikh U\v cebnykh Zavedeni\u\i, Matematika {\bf 1962}, 
no.3, 164--176 [Russian]; MR {\bf 25}\#3105; R\v ZMat 1963,1A220; Zbl {\bf 228}.20062.

\bibitem{Sch2} 
Boris M. Schein, On the theory of generalized groups and generalized grouds. Teoriya Polugrupp i Ee Prilo\v zeniya, Saratov University Press, Saratov {\bf 1}(1965), 286-324; MR {\bf 35}\#283; R\v ZMat 1966,12A253; Zbl {\bf 247}.20060 [in Russian. See English translation in: Translations of the American Mathematical Society (2) {\bf 113}(1979), 89--122.]

\bibitem{Sch3} Boris M. Schein, Stationary subsets, stabilizers and transitive representations of semigroups, Dissertationes Mathematic\ae, Warsaw {\bf 77}(1970), 41 pp.; MR {\bf 44}\#4129; R\v ZMat 1971,6A183; Zbl {\bf 219}, 139.

\bibitem{Sch4} Boris M. Schein, Semigroups in which every transitive representation by functions is a representation by invertible functions. Izvestiya Vys\v sikh U\v cebnykh Zavedeni\u\i. Matematika {\bf 1973}, no. 7, 112--121. [In Russian. See English translation in: Translations of the American Mathematical Society (2) {\bf 139}(1988), 165--176.]

\bibitem{Sch5} Boris M. Schein, Completions, translational hulls and ideal extensions of inverse semigroups. Czechoslovak Mathematical Journal {\bf 23}(1973), no.4, 575-610; MR {\bf 48}\#4166; R\v ZMat 1974,5A205; Zbl {\bf 273}.20047.

\bibitem{Sch6} Boris M. Schein, Noble inverse semigroups with bisimple core. Semigroup Forum 
{\bf 36}(1987), no.2, 175-178.

\bibitem{Sch7}
Boris M. Schein, Infinitesimal elements and transitive representations of inverse semigroups. Proceedings of the International Symposium on the Semigroup Theory and Its Related Fields, Held at Suekawa Memorial Hall, Ritsumeikan University, Kyoto, Japan, August 30-September 1, 1990 (M. Yamada and H. Tominaga, Editors), Shimane University, Matsue, Japan, 1990, 197-203.; MR {\bf 92c}:20119; Zbl {\bf 0729}.20032.

\bibitem{Sch8} Boris M. Schein, Cosets in groups and semigroups. Proceedings of the Conference on Semigroups with Applications (Oberwolfach, 21-28 July 1991), World Scientific Publishing Co., Singapore, 1992, 205-221. MR {\bf 94f}:20116

\bibitem{Sch9}  Boris M. Schein, Semigroups of cosets of semigroups: variations on a Dubreil theme. Collectanea Mathematica {\bf 46}(1995), nos. 1--2, 171-182; MR {\bf 96j}:20081; Zbl {\bf 848}.20058.

\bibitem{Stoll} R. R. Stoll, Representation of finite simple semigroups. Duke Mathematical Journal {\bf 11}(1944), 251Ð265; MR {\bf 5}, 229). 

\end{thebibliography}
\end{document}